\documentclass[reqno,twoside,10pt]{article} 

\usepackage{amsmath,amsthm,amssymb,amsfonts,mathrsfs,color,hyperref,xcolor}
\usepackage{bbm}
\usepackage{amsbsy}
\usepackage{latexsym}
\usepackage[final]{graphicx}
\usepackage[utf8]{inputenc}

\theoremstyle{plain}
\begingroup
\newtheorem{thm}{Theorem}[section]
\newtheorem{lem}[thm]{Lemma}
\newtheorem{prop}[thm]{Proposition}
\newtheorem{cor}[thm]{Corollary}
\newtheorem{example}[thm]{Example}

\endgroup

\theoremstyle{definition}
\begingroup
\newtheorem{defn}[thm]{Definition}
\newtheorem{rem}[thm]{Remark}
\endgroup

\graphicspath{{Pictures/}}

\numberwithin{equation}{section}

\newcommand{\res}{\mathop{\hbox{\vrule height 7pt width .5pt depth 0pt
\vrule height .5pt width 6pt depth 0pt}}\nolimits}

\newcommand{\N}{\mathbb N} 
\newcommand{\R}{\mathbb R}

\newcommand{\Mmn}{{\mathbb M}^{m{\times}n}}
\newcommand{\Ms}{{\mathbb M}^{n{\times}n}_{\rm sym}}
\newcommand{\Mstwo}{{\mathbb M}^{2{\times}2}_{\rm sym}}

\newcommand{\Mskew}{{\mathbb M}^{n{\times}n}_{\rm skew}}

\newcommand{\wto}{\rightharpoonup}
\newcommand{\e}{\varepsilon}

\newcommand{\LL}{{\mathcal L}}
\newcommand{\HH}{{\mathcal H}}
\newcommand{\M}{{\mathcal M}}
\newcommand{\C}{{\mathcal C}}

\newcommand{\Svar}[1]{{\left|#1\right|_{\rm S}}}
\newcommand{\Ssym}[1]{{\left|#1\right|_{\rm sym-S}}}

\DeclareMathOperator{\tr}{tr}

\DeclareMathOperator{\dive}{div}

\DeclareMathOperator{\diam}{{\rm diam}}

\let\O=\Omega

\newcommand{\fla}[1]{{\color{red}#1}}

\begin{document}
 \begin{center}
 	{\Large Piecewise rank-one approximation of vector fields with measure derivatives}\\[5mm]
 	{\today}\\[5mm]
 	Jean-Fran\c cois Babadjian$^{1}$ and Flaviana Iurlano$^{2}$\\[2mm]
 	{\em $^{1}$ Universit\'e Paris-Saclay, CNRS,  Laboratoire de math\'ematiques d'Orsay, 91405, Orsay, France
 	}\\[1mm]
 	{\em $^{2}$ Sorbonne Université, CNRS, Université Paris Cité, Laboratoire Jacques-Louis Lions, 75005 Paris, France}
 	\\[3mm]
 	\begin{minipage}[c]{0.8\textwidth}
 	{\small \textbf{Abstract.} 
	{This work adresses the question of density of piecewise constant (resp. rigid) functions in the space of vector valued functions with bounded variation (resp. deformation) with respect to the strict convergence. Such an approximation property cannot hold when considering the usual total variation in the space of measures associated to the standard Frobenius norm in the space of matrices. It turns out that oscillation and concentration phenomena interact in such a way that the Frobenius norm has to be homogenized into a (resp. symmetric) Schatten-1 norm which coincides with the Euclidean norm on rank-one (resp. symmetric) matrices. By means of explicit constructions consisting of the superposition of sequential laminates, the validity of an optimal approximation property is established at the expense of endowing the space of measures with a total variation associated with the homogenized norm in the space of matrices.}}

 	\end{minipage}
 \end{center}
%

\date{\today}
%
%
%

\section{Introduction}
Functions of bounded variation ($BV$) and of bounded deformation ($BD$) have shown to be the right mathematical setting of many applied problems, including problems from image segmentation, optimization, fracture {and} plasticity. This has led in the last decades to a many-sided and fine analysis of their properties. One of the first questions, motivated both by theoretical usefulness and by numerical application, is that of the approximation of $BV$ and $BD$ functions by {a more restricted class of} functions. The choice of the convergence and the regularity of the approximants may change with the context \cite{BraidesChiadoPiat,cortesani1997strong,CortesaniToader,AmarDeCicco,BellettiniChambolleGoldman,KristensenRindler,DePhilippisFuscoPratelli,CFI2023,chambolle,chambolle05,iur12,ContiFocardiIurlano2017IntRepr,friedrich2018piecewise,ContiFocardiIurlano2019-Density,ChambolleCrismale,Crismale}.

A classical result in the theory of scalar $BV$-functions states that any function $u \in BV(\O)$ in an open set $\O \subset \R^n$ can be approximated by a sequence of simple $BV$-functions with respect to the strict convergence, that is the strong $L^1$-convergence of functions together with the convergence of the total variation of their derivatives. The proof rests on the Fleming-Rishel coarea formula, and consequently does not easily apply to the vectorial setting, where the total variation is intended as the extension to measures of the Frobenius norm of matrices. In fact, it turns out that it is not possible to strictly approximate vector-valued $BV$ functions by piecewise constant {ones}, and not even if the convergence of the total variation of the derivatives is replaced by the convergence of their area-variation, or of any other strongly convex variation coinciding with the standard variation on singular rank-one measures (see respectively \cite{ambrosio1990}, \cite[Remark 22]{AABU}, \cite[Proposition 4]{KristensenRindler}, or Example \ref{ex} below).

The first object of this work is to show that such a density result holds for bounded variation vector fields, once one replaces the convergence of the total variation of the derivatives by the convergence of their Schatten variation, which extends to measures the Schatten-1 norm of matrices. The latter is given by the sum of the singular values of a matrix and therefore coincides with the Frobenius norm on rank-one matrices, but it is not strictly convex (we refer to Section \ref{sec:notation} for the explicit definition). 

\medskip

The second object of this work is to establish a similar result in the $BD$ framework. In this case the approximants are piecewise rigid motions and the variation is a nontrivial adaptation of the Schatten one, now applying to eigenvalues in place of singular values, and distinguishing between their sign. In particular it coincides with the standard variation on singular rank-one-symmetric measures (see again Section \ref{sec:notation}).

In fact, in the vectorial setting, it turns out that oscillation effects interact with concentrations, leading to homogenization of the norm with respect to which the total variation is taken. The norm has therefore to be changed into a homogenized one which coincides with the usual (Frobenius) norm on ``elementary'' objects. This homogenization phenomenon was already observed in earlier homogenization works in several contexts \cite{MT,KS,FraMur,AK1,AK2,AF,AL,A,Bouchitte,BIR1,BIR2}, {\it a priori} unrelated to the present study. A density result in the spirit of the present work has been recently established in \cite{AABU,ABC} in the context of functions with bounded Hessian, involving the corresponding Schatten variation of the Hessian matrix. Let us also mention that relaxation of the total variation of locally simple functions has been used in \cite[Example 3.5]{ambrosio1990} and \cite[Section 3.2]{BrenaNobiliPasqualetto} to give an alternative notion of $BV$-maps. A relaxation result in the same spirit, but in the $BD$ framework, is obtained in \cite[Theorem 2.1]{BDFV} for a functional defined on piecewise smooth functions with divergence in $L^2$. The link between the corresponding relaxation formula and the symmetric Schatten-$1$ norm is made explicit in \cite[Proposition 3.7]{BIR1}.

\medskip

These results lead to the following alternative expression of the Schatten-$1$ norm as well as of its symmetric version, respectively by means of convexification formulas (see Lemmas \ref{lem:Schatten1} and \ref{lem:Schatten1-sym}) 
\[\Svar{A}=(|\cdot|+I_{\Lambda_\text{curl}})^{**}(A), \quad \text{for}\quad A\in \Mmn,\] 
\[\Ssym{A}=(|\cdot|+I_{\Lambda_\text{curl curl}})^{**}(A), \quad \text{for}\quad A\in \Ms,\] 
where $|\cdot|$ is the Frobenius norm and, given a set $E$, $I_E$ stands for the indicator function of $E$ (taking the value $0$ in $E$ and $+\infty$ outside), and the double-star stands for the bi-convex conjugate (the double Legendre transform). In the previous expression $\Lambda_\text{curl}$ (resp. $\Lambda_\text{curl curl}$) denotes the wave cone (see \cite{DPR}) associated to the ${\rm curl}$ (resp. ${\rm curl\, curl}$) operator which is related to $BV$ (resp. $BD$) functions through the formula
${\rm curl}(Du)=0$ in $\mathcal D'(\O;\R^m)$ for all $u \in BV(\O;\R^m)$ (resp. ${\rm curl\, curl}(Eu)=0$ in $\mathcal D'(\O;\R^n)$ for all $u \in BD(\O)$). It translates pieces of information of the differential operator ${\rm curl}$ (resp. ${\rm curl\, curl}$) in the Fourier space. The theory of compensated compactness (see \cite{Tartar}) makes it explicit
$$\Lambda_\text{curl}=\{a \otimes b: \; a \in \R^m, \, b \in \R^n\}, \quad \Lambda_\text{curl curl}=\{a \odot b: \; a, b \in \R^n\}.$$

In view of the previous discussion, the results presented in this paper are somehow natural, but we were not able to find them in the literature. Our approach is direct, indeed the construction only relies on explicit {superposition of} sequential laminates, constructed on the basis of these last convexification formulas. Using standard approximation results of $BV$ or $BD$ functions by smooth ones, we can easily reduce to the case where the initial function is continuous, piecewise affine (on a fixed partition of $\R^n$ made of $n$-simplexes) and compactly supported. Then, working separately inside each $n$-simplex, the overall idea consists in considering laminates in suitable directions related to the singular values (eigenvalues in the $BD$ case, resp.) of the (symmetric) gradient of $u$, which is a constant matrix. By construction, such sequential laminates are piecewise rank-one (symmetric), hence the Frobenius variation and the (symmetric) Schatten variation of the derivatives coincide. It turns out that the optimal directions strongly depend on the singular vectors (on the sign of the eigenvalues and then on a combination of the eigenvectors, resp.), which leads to a loss of isotropy and a change of norm. In the limit the relaxed variations appear, that is, the Schatten variation in the $BV$ case and its symmetric version in the $BD$ case.

By lower semicontinuity, $\Svar{\cdot}$ and $\Ssym{\cdot}$ are the only variations coinciding with the standard variation respectively on rank-one and symmetric rank-one objects for which an approximation of the type above holds, see Propositions \ref{unique} and \ref{prop:uniqueBV} for more details.

\section{Notation}\label{sec:notation}

\subsection{Matrices}

Let $\Mmn$ be the set of all real $m \times n$ matrices. We recall that, given two matrices $A$, $B \in \Mmn$, the Frobenius scalar product is defined by $A : B = \tr(A^T B)$
and the associated Frobenius norm is given by $|A|:=\sqrt{A:A}$. If $A \in \Mmn$, the singular values $s_1,\ldots,s_n\geq 0$ of $A$ are defined as the eigenvalues of the (symmetric) matrix $\sqrt{A^TA} \in \Ms$. For every $p \in [1,\infty]$, the Schatten $p$-norm of $A$ is defined by
$$|A|_{S,p}:=\|(s_1,\ldots,s_n)\|_{\ell^p}.$$
Note that for $p=1$ (the case of interest in the sequel), we set $\Svar{A}:=|A|_{S,1}$ and we have
$$\Svar{A}=\tr\left(\sqrt{A^T A}\right),$$
while for $p=\infty$,
$$|A|_{S,\infty}= \varrho\left(\sqrt{A^T A}\right),$$
where $\varrho$ denotes the spectral radius.

Given $a \in \R^m$ and $b \in \R^n$, we denote by $a\otimes b:=ab^T \in \Mmn$ the tensor product between $a$ and $b$. We observe that 
\begin{equation}\label{eq:S1=}
\Svar{a \otimes b}=|a||b|=|a \otimes b|.
\end{equation}

\medskip

One has the following convexification formula for the $1$-Schatten norm.
\begin{lem}\label{lem:Schatten1} 
For all $A \in \Mmn$, we have
$$\Svar{A}=\sup\{\psi : \Mmn \to \R \text{ convex such that }\psi(a \otimes b) \leq |a \otimes b| \text{ for all }a \in \R^m, \, b \in \R^n\}.$$
\end{lem}

\begin{proof}
Let us denote by $h(A)$ the above supremum. Since $\Svar{\cdot}$ is convex and satisfies \eqref{eq:S1=}, we have that $h(A) \geq \Svar{A}$. We are thus back to show the reverse inequality. For all $A \in \Mmn$, let
$$H(A):=
\begin{cases}
|a \otimes b| & \text{ if }A=a \otimes b,\\
+\infty & \text{ otherwise.}
\end{cases}$$
Let us compute the convex conjugate of $H$. For all $X \in \Mmn$, we have
\begin{eqnarray*}
H^*(X) & = & \sup_{t>0} \sup_{|a|=|b|=1} t \{X:(a \otimes b) - |a \otimes b|\}\\
& = & \sup_{t>0} \sup_{|a|=|b|=1} t \{(Xa)\cdot b) -1\}\\
 & = & \sup_{t>0} \sup_{|a|=1} t \{|Xa|-1\}\\
  & = & \sup_{t>0} t \left\{\sqrt{\varrho(X^T X)}-1\right\}=I_{B}(X),
\end{eqnarray*}
where $B$ is the closed unit ball in $\Mmn$ with respect to the Schatten $\infty$-norm. As a consequence, $h=H^{**}$ (which is the convex envelope of $H$) coincides with the support function of $B$, which is nothing but the Schatten $1$-norm (see e.g. \cite[Proposition 1]{AU}).
\end{proof}

\medskip

We denote by $\Ms$ and $\Mskew$ the subspaces of $\mathbb M^{n \times n}$ of symmetric and skew-symmetric matrices, respectively. If $A \in \Ms$ is symmetric and $\lambda_1,\ldots,\lambda_n \in \R$ are the eigenvalues of $A$, then $s_i=|\lambda_i|$. We consider the function $\Ssym{\, \cdot\, } :\Ms \to \R^+$ introduced in \cite{BIR1} and defined, for all $\xi \in \Ms$, by
$$\Ssym{A}:=\sqrt{\frac12\left(\sum_{i=1}^n |\lambda_i|\right)^2 + \frac12\left(\sum_{i=1}^n \lambda_i \right)^2}=\sqrt{\frac12 \Svar{A}^2 + \frac12(\tr(A))^2}.$$
This function is a norm over $\Ms$, in particular, it is a continuous, Lipschitz and positively one-homogeneous function. It will be referred to as the symmetric Schatten $1$-norm and it can alternatively be written as
\begin{equation}\label{posneg}
\Ssym{A}=\sqrt{\left(\sum_{i=1}^l \lambda_i\right)^2 + \left(\sum_{i=l+1}^n \lambda_i \right)^2},
\end{equation}
where $\lambda_1\leq\dots\leq\lambda_l\leq0<\lambda_{l+1}\leq\dots\leq\lambda_n$. In dimension $n=2$, it further reduces to
$$\Ssym{A}=\sqrt{|A|^2+2(\det A)^+} \quad \text{ for all }A \in \Mstwo.$$
If $a$ and $b \in \R^n$, we denote by $a \odot b=(a\otimes b +b \otimes a)/2 \in \Ms$ the symmetric tensor product between $a$ and $b$. Since, by \cite[Lemma 2.1]{BIR1}, such a matrix has rank at most $2$ and both nonzero eigenvalues have opposite signs, we infer that 
\begin{equation}\label{eq:hsym}
\Ssym{a \odot b}=|a \odot b|.
\end{equation}

One has a similar convexification formula for this symmetric $1$-Schatten norm. The proof directly follows from \cite[Proposition 3.7]{BIR1}.
\begin{lem}\label{lem:Schatten1-sym} 
For all $A \in \Ms$, we have
$$\Ssym{A}=\sup\{\psi : \Ms \to \R \text{ convex such that }\psi(a \odot b) \leq |a \odot b| \text{ for all }a, \, b \in \R^n\}.$$
\end{lem}

\subsection{Measures}

The Lebesgue measure in $\R^n$ is denoted by $\mathcal L^n$, and the $(n-1)$-dimensional Hausdorff measure by $\mathcal H^{n-1}$. If $\Omega \subset \R^n$ is an open set and $X$ is an Euclidean space, we denote by $\mathcal M(\O;X)$ the space of $X$-valued bounded Radon measures in $\O$.
Fixed a norm $\|\cdot\|$ on $X$, we can define the variation of $\mu\in\mathcal M(\O;X)$ with respect to $\|\cdot\|$ as 
\[\|\mu\|(A):=\sup\sum_{i=1}^k\|\mu(A_i)\|,\]
for all Borel set $A\subset\Omega$, where the supremum is taken over all finite Borel partitions $A=\bigcup_{i=1}^k A_i$ of $A$. 
In particular, $\|\mu\|(\O)$ is a norm on $\mathcal M(\O;X)$. Frobenius variation will be denoted by $|\cdot|$.
In the sequel, we will consider the cases $X=\mathbb M^{m \times n}$, $X=\Ms$ and $X=\R$. When $X=\R$, we simply write $\mathcal M(\O)$ instead of $\mathcal M(\O;\R)$.
	
Let $\mu \in \mathcal M(\O;X)$ and $f:X \to [0,+\infty)$ be a convex, positively one-homogeneous function.  Using the theory of convex functions of measures developed in \cite{GS}, we introduce the nonnegative finite Borel measure $f(\mu)$, defined by 
$$f(\mu):=f\left(\frac{d \mu}{d |\mu|}\right)|\mu|\,,$$
where $\frac{d \mu}{d |\mu|}$ stands for the Radon-Nikod\'ym derivative of $\mu$ with respect to $|\mu|$.
	
\subsection{Functional spaces}

The space of (vector valued) functions of bounded variation is defined by
$$BV(\O;\R^m)=\{u \in L^1(\O;\R^m): \; Du \in \mathcal M(\O;\Mmn)\}.$$
We refer to \cite{AFP} for notation and general properties of that space. We further consider the space of functions of bounded deformation introduced in \cite{Suquet} defined by
$$BD(\Omega)=\{u \in L^1(\Omega;\R^n) : \; Eu:=(Du+Du^T)/2\in \mathcal M(\Omega;\Ms)\}.$$
We refer to \cite{ACDM,Temam} for a general treatment of that space.

\section{Symmetric Schatten mass approximation of functions of bounded deformation}

\begin{defn}
A function $u:\R^n \to \R^n$ is a {\it piecewise rigid body motion} if there exist a countable family $\{P_i\}_{i \in \N}$ of polyhedral sets with pairwise disjoint interiors such that  $\R^n= \bigcup_{i\in \N} P_i$, skew symmetric matrices $\{A_i\}_{i \in \N} \subset \Mskew$ and vectors $\{b_i\}_{i \in \N} \subset \R^n$ such that for all $i \in \N$,
$$u(x)=A_i x + b_i \quad \text{ for all }x \in P_i.$$
{We denote by $PR(\R^n)$ the space of all such piecewise rigid body motions.}
\end{defn}

\begin{rem}\label{rem:pwcr-jump1}
Note that if {$u \in PR(\R^n)$}, then its symmetric gradient $Eu$ is a pure jump measure concentrated on $\bigcup_{i,j \in \N} (\partial P_i \cap \partial P_j)$. Moreover,
$$Eu\res (\partial P_i \cap \partial P_j)= \left((u|_{P_j})^+- (u|_{P_i})^-\right) \odot \nu \HH^{n-1} \res (\partial P_i \cap \partial P_j),$$
where $\nu$ is the unit vector orthogonal to $\partial P_i \cap \partial P_j$ oriented from $P_i$ to $P_j$, and $(u|_{P_j})^+$ (resp. $ (u|_{P_i})^-$) is the trace of $u|_{P_j}$ (resp. $ u|_{P_i}$) on $\partial P_i \cap \partial P_j$. In particular, \eqref{eq:hsym} ensures that $\Ssym{Eu}=|Eu|$ as measures in $\R^n$.
\end{rem}

Our first main result is the following approximation of functions of bounded deformation.

\begin{thm}\label{thm:BD}
Let $\O$ be a bounded open set of $\R^n$ with Lipschitz boundary and $u \in BD(\O)$. Then, there exists a sequence $\{u_k\}_{k \in \N}$ {in $PR(\R^n)$} such that
$$
\begin{cases}
u_k \to u \quad \text{strongly in }L^1(\O;\R^n),\\
\Ssym{Eu_k}(\O) \to \Ssym{Eu}(\O).
\end{cases}
$$
\end{thm}

The previous approximation result holds for the symmetric Schatten mass $\Ssym{Eu}$ of $Eu$, and not for the usual total variation mass $|Eu|$ as the following counterexample shows (see also \cite[Remark 22]{AABU}, \cite[Proposition 4]{KristensenRindler}, or Example \ref{ex} below for the Schatten $1$-norm).

\begin{example}\label{ex:sym}{\rm
Let $n=2$, $\O=Q=(0,1)^2$ is the unit cube and $u(x)=x$, so that $Eu={\rm Id}$. Assume that there exists a sequences of piecewise rigid body motions $\{u_k\}_{k \in \N}$ such that $u_k \to u$ strongly in $L^1(Q;\R^2)$ and $|Eu_k|(Q) \to |Eu|(Q)$. Since $\Ssym{\cdot}$ is a continuous and positively one-homogeneous function, Reshetnyak continuity Theorem (\cite[Theorem 2.39]{AFP}) ensures that $\Ssym{Eu_k}(Q) \to \Ssym{Eu}(Q)$. But since $u_k$ is piecewise rigid, Remark \ref{rem:pwcr-jump1} shows that
$$\frac{dEu_k}{d|Eu_k|}=\alpha_k \odot \beta_k$$
for some Borel functions $\alpha_k$, $\beta_k:\R^2 \to \R^2$ such that $|\alpha_k \odot \beta_k|=1$ and
$$\Ssym{Eu_k}=\Ssym{\frac{dEu_k}{d|Eu_k|}}|Eu_k|=\Ssym{\alpha_k \odot \beta_k}|Eu_k|.$$
But since by \eqref{eq:hsym}, $\Ssym{\alpha_k \odot \beta_k}=|\alpha_k \odot \beta_k|=1$, then $\Ssym{Eu_k}=|Eu_k|$ and thus $\Ssym{Eu}(Q)=|Eu|(Q)$. This is however not possible since a straightforward computation shows that $|Eu|(Q)=|{\rm Id}|=\sqrt{2}$ while 
$\Ssym{Eu}(Q)=\Ssym{{\rm Id}}=2$.
} 
\end{example}

\begin{rem}
A similar argument  would show that there is no sequence of piecewise rigid functions $u_k$ such that $u_k\to u$ strongly in $L^1(\O;\mathbb{R}^n)$ and $Eu_k$ converges to $Eu$ in mass with respect to a strictly convex norm. This follows from the fact that convergence in mass with respect to a strictly convex norm implies convergence in mass with respect to the Frobenius norm by Reshetnyak Continuity Theorem (see for instance \cite[Theorem 1.1]{Spector} and \cite{LM,Ret}. Note that weak* convergence in $\M(\O;\Ms)$ together with the convergence of the total variation associated to any norm in $\Ms$ implies weak* convergence in $[\C_b(\O;\Ms)]'$).
\end{rem}

This restriction of the admissible norms (for which such an approximation result is valid) is actually much stronger. Indeed, our Theorem \ref{thm:BD} is optimal in the sense that the symmetric Schatten-1 norm is the only possible norm coinciding with the Frobenius one on rank-one symmetric matrices for which such a density result holds.

\begin{prop}\label{unique}
Let $N$ be a norm over $\Ms$ be such that $N(a \odot b)=|a \odot b|$ for all $a$, $b \in \R^n$. Assume for all bounded open set $\O \subset \R^n$ with Lipschitz boundary and all $u \in BD(\O)$, there exists a sequence $\{u_k\}_{k \in \N}$ in $PR(\R^n)$ such that
$$
\begin{cases}
u_k \to u \quad \text{strongly in }L^1(\O;\R^n),\\
N(Eu_k)(\O) \to N(Eu)(\O)
\end{cases}
$$
Then $N=\Ssym{\cdot}$.
\end{prop}

\begin{proof}
We first observe that the proof of \cite[Proposition 3.7]{BIR1} (with the choice of parameters $\alpha=\kappa=1$, $\lambda_w=0$ and $\mu_w=1/2$) ensures that $N \leq \Ssym{\cdot}$, so that we are back to show the reverse inequality. 

Let $A \in \Ms$ and consider the linear map $u(x)=Ax$. By assumption, there is a sequence $\{u_k\}_{k \in \N}$ in $PR(\R^n)$ such that $u_k \to u$ strongly in $L^1(\O;\R^n)$ and $N(Eu_k)(\O) \to N(Eu)(\O)$. This implies in particular that $Eu_k \wto Eu$ weakly* in $\M(\O;\Ms)$. Since by Remark \ref{rem:pwcr-jump1}, $\frac{dEu_k}{d|Eu_k|}$ is a rank-one symmetric matrix, it results from \eqref{eq:hsym} that
\begin{eqnarray*}
N(Eu)(\O) & = & \lim_{k \to \infty} N(Eu_k)(\O) = \lim_{k \to \infty} |Eu_k|(\O)\\
& = & \lim_{k \to \infty} \Ssym{Eu_k}(\O) \geq \Ssym{Eu}(\O),
\end{eqnarray*}
where we used Reshetnyak's lower semicontinuity theorem in the last inequality (see \cite[Theorem 2.38]{AFP}). Recalling that $Eu=A$, we get that $N(A) \geq \Ssym{A}$.
\end{proof}

An immediate consequence of Theorem \ref{thm:BD} is the following relaxation result, in the spirit of \cite{BDFV}.
\begin{cor}\label{relBD}
Let $\Omega\subset\R^n$ be open, bounded, with Lipschitz boundary. For all $u \in L^1(\O;\R^n)$, define
$$\overline F(u):=\inf\left\{\liminf_{k \to \infty} |Eu_k|(\O) : \quad u_k \in PR(\R^n), \; u_k \to u \text{ in }L^1(\O;\R^n)\right\}.$$
Then
$$\overline F(u)=
\begin{cases}
\Ssym{Eu}(\Omega) & \text{ if }u \in BD(\O),\\
+\infty & \text{ otherwise}.
\end{cases}$$
\end{cor}

The proof of the lower bound inequality is an immediate consequence of Reshetnyak lower semicontinuity Theorem (see \cite[Theorem 2.38]{AFP}) together with the fact that $|Eu|=\Ssym{Eu}$ as measures if $u \in PR(\R^n)$. The upper bound is a direct consequence of Theorem \ref{thm:BD}.

\begin{proof}[Proof of Theorem \ref{thm:BD}]
{\bf Step 1}. Let us show that there is no loss of generality to suppose that $u$ is a continuous, piecewise affine function compactly supported in $\R^n$.

Since $\O$ has Lipschitz boundary, it follows from \cite[Theorem 3.2]{B} that $u \in BD(\O)$ has a trace on $\partial\O$, denoted by $\gamma(u)$, which belongs to $L^1(\partial\O;\R^n)$. Gagliardo's Theorem then ensures the existence of $v \in W^{1,1}(\R^n \setminus \overline\O;\R^n)$, compactly supported in $\R^n$, such that $\gamma(v)=\gamma(u)$ on $\partial\O$. Setting $\tilde u:=u {\bf 1}_\O + v {\bf 1}_{\R^n \setminus \overline \O}$ yields $\tilde u \in BD(\R^n)$ with
$$E\tilde u= Eu \res \O + e(v)\LL^n \res (\R^n \setminus \overline \O).$$
Let $\{\eta_\e\}_{\e>0}$ be a standard family of mollifiers and set 
$$\tilde u_\e:=\tilde u*\eta_\e \in \C^\infty_c(\R^n;\R^n).$$
Standard properties of mollifiers imply that $\tilde u_\e \to \tilde u$ strongly in $L^1(\R^n;\R^n)$, hence by lower semicontinuity of the total variation
$$|E\tilde u|(\R^n) \leq \liminf_{\e \to 0}|E\tilde u_\e|(\R^n).$$
Moreover, by \cite[Theorem 2.2]{AFP}, we also have that $E\tilde u_\e=(E\tilde u)*\eta_\e$ and $|E\tilde u_\e|(\R^n) \leq |E\tilde u|(\R^n)$. As a consequence, $|E\tilde u_\e|(\R^n) \to |E\tilde u|(\R^n)$ and Reshetnyak continuity Theorem (see (\cite[Theorem 2.39]{AFP})) shows that $\Ssym{E\tilde u_\e}(\R^n) \to \Ssym{E\tilde u}(\R^n)$. In particular, since $|E\tilde u|(\partial\O)=0$, we also get that $|E\tilde u_\e|(\O) \to |E\tilde u|(\O)$, and another application of the Reshetnyak continuity Theorem yields $\Ssym{E\tilde u_\e}(\O) \to \Ssym{E\tilde u}(\O)$.

The previous discussion shows that there is no loss of generality to assume that $u \in \C^\infty_c(\R^n;\R^n)$. 

\medskip
We next fix a small parameter $\delta>0$ and consider a triangulation of $\R^n$ into $n$-simplexes $\{T_i\}_{i \in \N}$ with pairwise disjoint interiors and such that, for all $i \in \N$, $\diam(T_i) \leq \delta$ and $\LL^n(T_i) \geq c\delta^n$ for some $c>0$. Let $\hat u_\delta$ be the continuous piecewise affine function which is the Lagrange interpolation of the values of $u$ at the nodes of the triangulation. Standard finite element estimates show that $\hat u_\delta \to u$ strongly in $W^{1,1}(\R^n;\R^n)$, hence in particular we get that $\Ssym{E\hat u_\delta}(\O) \to \Ssym{Eu}(\O)$.

\bigskip

\noindent {\bf Step 2.}  Assume that $T$ is an $n$-simplex and $u(x)=Ax+b$ for some $A \in \mathbb M^{n \times n}$ and $b \in \R^n$. We now modify $u$ by adapting and extending the construction exhibited in \cite[Section 3.1]{BIR1}. We denote by
$$A^{\rm sym}=\frac{A+A^T}{2} \in \mathbb M^{n \times n}_{\rm sym}, \quad A^{\rm skew}=\frac{A-A^T}{2}\in \mathbb M^{n \times n}_{\rm skew}$$
so that $A=A^{\rm sym}+A^{\rm skew}$. The matrix $A^{\rm sym}$ being symmetric, we can consider its spectral decomposition 
$$A^{\rm sym}=\sum_{i=1}^n \lambda_i e_i \otimes e_i,$$
here $\lambda_1,\ldots,\lambda_n \in \R$ are the eigenvalues and $e_1,\ldots,e_n$ are the associated eigenvectors of $A^{\rm sym}$ such that $\{e_1,\ldots,e_n\}$ {forms} an orthonormal basis of $\R^n$.

We argue by induction on the dimension $n$ to show the following property: there exists a sequence $\{v_k^T\}_{k \in \N}$ of piecewise constant functions $v_k^T:\R^n\to\R^n$ such that
\begin{equation}\label{eq:induction}
\begin{cases}
\text{$v_k^T(x) \to A^{\rm sym}x$ uniformly with respect to $x \in \R^n$,}\\
Ev_k^T\ll \mathcal{H}^{n-1}\res L_k,\\
 \Ssym{Ev^T_k}(T) \to \Ssym{Eu}(T),
\end{cases}
\end{equation}
where $L_k$ is a countable union of $(n-1)$-dimensional affine subspaces of $\R^{n}$, finitely many thereof intersecting $T$.

Once \eqref{eq:induction} is established, we set $u^T_k(x)=v^T_k(x)+A^{\rm skew}x+b$ which defines a piecewise rigid body motion in $T$, and satisfies $u^T_k \to u$ uniformly in $T$ and $\Ssym{Eu^T_k}(T)=\Ssym{Ev^T_k}(T) \to \Ssym{Eu}(T)$.

\medskip

{\it Step 2a.} Let us first assume that $n=2$. 

\medskip

$\bullet$ If $\det(A^{\rm sym})>0$, then the eigenvalues $\lambda_1$ and $\lambda_2$ have the same sign. It is immediate to check that $\Ssym{A^{\rm sym}}=|\lambda_1+\lambda_2|=|\lambda_1|+|\lambda_2|$. We introduce an auxiliary step function $s_k : \R \to \R$ defined by 
\begin{equation}\label{sk}s_k(t):=\sum_{i \in \mathbb Z} \frac{i}{k} {\bf 1}_{[\frac{i}{k},\frac{i+1}{k})}(t),
\end{equation}
which satisfies that $s_k(t) \to t$ uniformly in $\R$.
In that case, we set
$$v^T_k(x):=\lambda_1 s_k(x\cdot e_1) e_1+ \lambda_2 s_k(x \cdot e_2)e_2$$
so that $v^T_k(x) \to \lambda_1 (x\cdot e_1)e_1+\lambda_2 (x\cdot e_2)e_2=A^{\rm sym}x$ uniformly with respect to $x \in \R^2$. Next
$$Ev^T_k  =  \sum_{\ell \in \mathbb Z} \frac{\lambda_1e_1 \otimes e_1}{k} \HH^1\res \left\{x \in \R^2 : \; e_1\cdot x=\frac{\ell}{k}\right\}  +  \sum_{\ell \in \mathbb Z} \frac{\lambda_2e_2 \otimes e_2}{k} \HH^1\res  \left\{x \in \R^2 : \; e_2\cdot x=\frac{\ell}{k}\right\} .$$
Since both measures in the right-hand-side of the previous equality are mutually singular (because $e_1$ and $e_2$ are orthogonal) and $\Ssym{\lambda_i e_i \otimes e_i}=|\lambda_i|$, we infer that
$$\Ssym{Ev^T_k}  =  \sum_{\ell \in \mathbb Z} \frac{|\lambda_1|}{k} \HH^1\res \left\{x \in \R^2 : \; e_1\cdot x=\frac{\ell}{k}\right\}+  \sum_{\ell \in \mathbb Z} \frac{|\lambda_2|}{k} \HH^1\res  \left\{x \in \R^2 : \; e_2\cdot x=\frac{\ell}{k}\right\},$$
and thus, for every polyhedral set $P \subset \R^2$,
\begin{eqnarray*}
\Ssym{Ev^T_k}(P) & = & \sum_{\ell \in \mathbb Z} \frac{|\lambda_1|}{k} \HH^1\left( \left\{x \in P : \; e_1\cdot x=\frac{\ell}{k}\right\}\right)+  \sum_{\ell \in \mathbb Z} \frac{|\lambda_2|}{k} \HH^1\left(  \left\{x \in P : \; e_2\cdot x=\frac{\ell}{k}\right\}\right)\\
& \to & (|\lambda_1|+|\lambda_2|)\LL^2(P)=\Ssym{A^{\rm sym}}\LL^2(P)=\Ssym{Eu}(P).
\end{eqnarray*}

\medskip

$\bullet$ If $\det(A^{\rm sym})\leq 0$, by \cite[Lemma 2.1]{BIR1}, we can write $A^{\rm sym}=\alpha \odot \beta$ for some vectors $\alpha$, $\beta \in \R^2$. Note that, if $\alpha$ and $\beta$ are linearly dependent, we can assume that $\alpha=\beta$. 
We define
$$v^T_k(x):=\frac{\alpha |\beta|}{2}s_k\left(\frac{\beta\cdot x}{|\beta|} \right) + \frac{\beta|\alpha|}{2} s_k\left(\frac{\alpha\cdot x}{|\alpha|}\right)$$
so that $v^T_k(x) \to \frac{\alpha}{2}(\beta\cdot x) + \frac{\beta}{2}(\alpha\cdot x)=A^{\rm sym}x$ uniformly with respect to $x \in \R^2$. Computing the symmetric gradient yields
$$Ev^T_k  =  \sum_{\ell \in \mathbb Z} \frac{\alpha \odot \beta}{2k} \HH^1\res \left\{x \in \R^2 : \; \frac{\beta}{|\beta|} \cdot x=\frac{\ell}{k}\right\}  +  \sum_{\ell \in \mathbb Z} \frac{\alpha \odot \beta}{2k} \HH^1\res \left\{x \in \R^2 : \; \frac{\alpha}{|\alpha|} \cdot x=\frac{\ell}{k}\right\}.$$
As a consequence, we get that for every polyhedral set $P \subset \R^2$,
\begin{multline*}
\Ssym{Ev^T_k}(P)  =  \sum_{\ell \in \mathbb Z} \frac{\Ssym{\alpha \odot \beta}}{2k} \HH^1\left( \left\{x \in P : \;  \frac{\beta}{|\beta|}  \cdot x=\frac{\ell}{k}\right\}\right)\\
 +  \sum_{\ell \in \mathbb Z} \frac{\Ssym{\alpha \odot \beta}}{2k} \HH^1\left( \left\{x \in P : \; \frac{\alpha}{|\alpha|} \cdot x=\frac{\ell}{k}\right\}\right),
  \end{multline*}
hence $\Ssym{Ev^T_k}(P) \to \Ssym{\alpha \odot \beta} \LL^2(P)=\Ssym{Eu}(P)$.

Taking $P=T$, we have thus proved the validity of \eqref{eq:induction} for $n=2$.
\medskip

{\it Step 2b.} Let us assume now that there exists $n \geq 3$ such that \eqref{eq:induction} holds for all $d \in \{2,\ldots,n-1\}$.

\medskip 

$\bullet$ If $\lambda_i\geq 0$ (or similarly $\lambda_i\leq 0$) for all $i=1,\dots,n$, then we construct $v^T_k$ by laminating componentwise, that is 
\[v^T_k(x):=\sum_{i=1}^n\lambda_is_k(x\cdot e_i)e_i.\]
We immediately check as in Step 2a that $v^T_k$ converges to $x \mapsto A^{\rm sym}x$ uniformly in $\R^n$ and that 
$$\lim_{k \to \infty} \Ssym{Ev^T_k}(T)\to\sum_{i=1}^n|\lambda_i|\mathcal{L}^n(T)=\Ssym{Eu}(T).$$

\medskip 

$\bullet$ If all eigenvalues but one have the same sign, that is, $\lambda_1\leq\dots\leq \lambda_{n-1}\leq0< \lambda_n$ and $\lambda_1<0$ (or similarly $\lambda_1\geq\dots\geq \lambda_{n-1}\geq0 > \lambda_n$ and $\lambda_1>0$), we decompose the matrix $A$ as the sum of  rank-one symmetric matrices. More precisely, we set
$$A=\sum_{j=1}^{n-1} A^{(j)},$$
where 
$$A^{(j)}= \lambda_j e_j \otimes e_j+ \delta_j \lambda_n e_n \otimes e_n, \qquad \delta_j:=\lambda_j\left(\sum_{l=1}^{n-1}\lambda_l\right)^{-1}.$$
Note that $A^{(j)}=\alpha^{(j)} \odot \beta^{(j)}$ where
$$\alpha^{(j)}=-\sqrt{|\lambda_j|} e_j + \sqrt{\delta_j \lambda_n} e_n, \qquad \beta^{(j)}=\sqrt{|\lambda_j|} e_j + \sqrt{\delta_j \lambda_n} e_n.$$
We next define
$$v^T_k(x):=\sum_{j=1}^{n-1} \left[\frac{\alpha^{(j)} |\beta^{(j)}|}{2}s_k\left(\frac{\beta^{(j)}\cdot x}{|\beta^{(j)}|} \right) + \frac{\beta^{(j)}|\alpha^{(j)}|}{2} s_k\left(\frac{\alpha^{(j)}\cdot x}{|\alpha^{(j)}|}\right)\right].$$
Arguing as in Step 2a, we get that $v^T_k$ converges to $x \mapsto A^{\rm sym}x$ uniformly in $\R^n$ and 
$$\lim_{k \to \infty}\Ssym{Ev^T_k}(T)=\sum_{j=1}^{n-1}\Ssym{A^{(j)}}\mathcal{L}^n(T).$$
Since $A^{(j)}$ is a rank-one symmetric matrix, then \eqref{eq:hsym} implies that
\begin{eqnarray*}
\sum_{j=1}^{n-1}\Ssym{A^{(j)}} & = & \sum_{j=1}^{n-1}|A^{(j)}|= \sum_{j=1}^{n-1} \sqrt{\lambda_j^2 +\delta_j^2\lambda_n^2}\\
& = & \sqrt{\left(\sum_{l=1}^{n-1}\lambda_l\right)^2 + \lambda_n^2} = \Ssym{A^{\rm sym}}
\end{eqnarray*}
by \eqref{posneg}, hence $\Ssym{Ev^T_k}(T)\to \Ssym{Eu}(T)$.

\medskip 

$\bullet$ It remains to consider the case where $\lambda_1\leq\dots\leq \lambda_r<0<\lambda_{m+1}\leq\dots\leq \lambda_n$ with $2\leq r\leq m\leq n-2$.
 
We decompose $A^{\rm sym}$ as sum of lower rank symmetric matrices, in such a way that $\Ssym{\cdot}$ stays additive on such a decomposition. Indeed, let us write
\[A^{\rm sym}=\sum_{j={m+1}}^nA^{(j)},\]
where, for all $j=m+1,\ldots,n$,
\[A^{(j)}:=\lambda_j e_j\otimes e_j+a_j\sum_{i=1}^m  \lambda_i e_i \otimes e_i, \qquad a_j:= \lambda_j\left(\sum_{l=m+1}^n\lambda_l\right)^{-1}.\]	
Notice that $0< a_j\leq 1$ and $\sum_{j=m+1}^na_j=1$; moreover, $a_j\lambda_1\leq\dots \leq a_j\lambda_r< 0<\lambda_j$ and, since $r+1\leq m+1<n$, then $A^{(j)}$ has at least one zero eigenvalue. 
	
Let $\{\hat e_1,\ldots,\hat e_{{r+1}}\}$ be the canonical basis of $\R^{{r+1}}$. For all $m+1 \leq j \leq n$, we denote by 
$$\hat A^{(j)}:=a_j \sum_{i=1}^{{r}}  \lambda_i \hat e_i \otimes \hat e_i + \lambda_j \hat e_{{r+1}} \otimes \hat e_{{r+1}} \in \mathbb M^{({r+1}) \times ({r+1})}_{\rm sym}$$
and 
$$\hat T:=\left\{\hat x \in \R^{{{r}}+1} : \; \sum_{i=1}^{{r}} \hat x_i e_i + \hat x_{{{r}}+1} e_j \in T\right\},$$
which is a polyhedral subset of $\R^{{{r}}+1}$. 	
Using again that $m+1 <n$,  the inductive step ensures the existence of a sequence $\{\hat v_k^{(j)}\}_{k \in \N}$, {$\hat v_k^{(j)}:\R^{{{r}}+1}\to\R^{{{r}}+1}$}, of piecewise constant functions such that $\hat v_k^{(j)}(\hat x) \to   \hat A^{(j)} \hat x$ uniformly with respect to $\hat x \in \R^{{{r}}+1}$, $$\Ssym{E\hat v_k^{(j)}}(\hat T)\to \Ssym{\hat A^{(j)}}(\hat T)$$ as $k\to\infty$ and $E\hat v_k^{(j)}\ll \mathcal{H}^{{{r}}}\res \hat L^{(j)}_k$, where 
$\hat L_k^{(j)}$ is a countable union of ${{r}}$-dimensional affine subspaces of $\R^{{{r}}+1}$, finitely many thereof intersecting $\hat T$.
 
Hence, setting 
$$v_k^{(j)}(x):=\sum_{l=1}^{{r}} \left[\hat v_k^{(j)} \left(\sum_{i=1}^{{r}} x_i \hat e_i + x_j \hat e_{{{r}}+1}\right) \cdot \hat e_l \right] e_l+ \left[\hat v_k^{(j)} \left(\sum_{i=1}^{{r}} x_i \hat e_i + x_j \hat e_{{{r}}+1}\right) \cdot \hat e_{{{r}}+1} \right] e_j$$
and 
$$v_k^T(x):=\sum_{j=m+1}^n v_k^{(j)}(x) ,$$
we get 
\begin{eqnarray*}
v_k^T(x) & \to &\sum_{j=m+1}^n \left\{ \sum_{l=1}^{{r}} \left[\hat A^{(j)}\left(\sum_{i=1}^{{r}} x_i \hat e_i + x_j \hat e_{{{r}}+1}\right)\cdot \hat e_l \right]e_l\right.\\
&&\hspace{5cm} \left. +\left[\hat A^{(j)}\left(\sum_{i=1}^{{r}} x_i \hat e_i + x_j \hat e_{{{r}}+1}\right)\cdot \hat e_{{{r}}+1} \right]e_j\right\}\\
 & = &  \sum_{j=m+1}^n\left\{ \sum_{l=1}^{{r}} \left[\sum_{i=1}^{{r}} \big(a_j \lambda_i x_i \hat e_i + \lambda_j x_j \hat e_{{{r}}+1}\big) \cdot \hat e_l \right] e_l\right.\\
&&\hspace{5cm} \left. +\left[\sum_{i=1}^{{r}} \big(a_j \lambda_i x_i \hat e_i + \lambda_j x_j \hat e_{{{r}}+1}\big) \cdot \hat e_{{{r}}+1} \right] e_j\right\}\\
  & = &  \sum_{j=m+1}^n\left( \sum_{i=1}^{{r}} a_j \lambda_i x_i e_i + \lambda_j x_j e_j\right)\\
 & = &   \sum_{i=1}^{\fla{r}} \lambda_i x_i e_i + \sum_{j=m+1}^n \lambda_j x_j e_j=A^{\rm sym}x
 \end{eqnarray*}
 uniformly with respect to $x \in \R^n$. Note also that the measure  $E v_k^{(j)}$ is concentrated on a countable union of $(n-1)$-dimensional affine subspaces of $\R^n$ of the form
 $$ L_k^{(j)}=\{x \in \R^n : \; (x_1,\ldots,x_r,x_j) \in \hat L_k^{(j)}\}.$$ 
 Since we have $\HH^{n-1}(L_k^{(j)} \cap L_k^{(l)})=0$ when $l\neq j$, we get that the measures $E v_k^{(l)}$ and $E v_k^{(j)}$ are concentrated on essentially disjoint sets and therefore,
\[\Ssym{Ev_k^T}(T)= \sum_{j=m+1}^n\Ssym{Ev^{(j)}_k}(T)\to\sum_{j=m+1}^n\Ssym{A^{(j)}}(T).\]
Since, 
\begin{eqnarray*}
\sum_{j=m+1}^n \Ssym{A^{(j)}}&=&\sum_{j=m+1}^n\sqrt{\left(\sum_{i=1}^{{r}} a_j\lambda_i\right)^2+\lambda_j^2}\\
&=&\sqrt{\left(\sum_{i=1}^{{r}}\lambda_i\right)^2+\left(\sum_{l=m+1}^n\lambda_l\right)^2}=\Ssym{A^{\rm sym}},
\end{eqnarray*}
we deduce that $\Ssym{Ev_k^T}(T) \to \Ssym{Eu}(T)$.

\medskip

\noindent {\bf Step 3.} Assume that $u$ is a continuous, piecewise affine and compactly supported function on a partition $\{T_i\}_{i \in \N}$ of $\R^n$ made of $n$-simplexes. Applying the construction of Step 2 in each simplex $T_i$, we set
$$u_k:=\sum_{i \in \N} u_k^{T_i} {\bf 1}_{T_i},$$
which defines a sequence $\{u_k\}_{k \in \N}$ in $PR(\R^n)$ such that $u_k \to u$ uniformly in $\R^n$ and $\Ssym{Eu_k}(T_i) \to \Ssym{Eu}(T_i)$ for all $i \in \N$. Note that, since $u$ has compact support in $\R^n$, we get that 
\begin{equation}\label{eq:diese1}
\#\{i \in \N : \; T_i \cap {\rm Supp}(u)\neq \emptyset\}<\infty, \quad \sup_{k \in \N}\#\{i \in \N : \; T_i \cap {\rm Supp}(u_k)\neq \emptyset\}<\infty.
\end{equation}

It remains to estimate the measure $\Ssym{Eu_k}$ on the common interface $S_{ij}=\partial T_i \cap \partial T_j$ of two adjacent simplexes $T_i$ and $T_j$. To this aim, we observe that, by construction, for all $i \in \N$,
$$\|u_k-u\|_{L^\infty(T_i;\R^n)} \leq C\frac{|A_i^{\rm sym}|}{k},$$
and thus, owing to \eqref{eq:diese1},
$$\|u_k-u\|_{L^\infty(\O;\R^n)} \leq \frac{C}{k}$$
for some constant $C>0$ independent of $k$. As a consequence of the continuity of $u$, we deduce that for all $x \in S_{ij}$
$$|u_k^+(x)- u_k^-(x)| \leq \frac{C}{k},$$
where $u_k^-=u^{T_i}_k|_{S_{ij}}$ and $u_k^+=u^{T_j}_k|_{S_{ij}}$ denote the one-sided traces of $u_k$ on both sides of $S_{ij}$. On the other hand, the jump formula yields
$$Eu_k \res S_{ij}= (u_k^+-u_k^-)\odot \nu \HH^{n-1}\res S_{ij},$$
where $\nu$ is the normal vector to $S_{ij}$ oriented from $T_i$ to $T_j$. As a consequence,
\begin{multline*}
\Ssym{Eu_k}(S_{ij})= \int_{S_{ij}} \Ssym{(u_k^+-u_k^-)\odot \nu}\, d \HH^{n-1} \\
= \int_{S_{ij}} |(u_k^+-u_k^-)\odot \nu| \, d\HH^{n-1} \leq  \frac{C}{k}\HH^{n-1}(S_{ij}) \to 0
\end{multline*}
as $k \to \infty$. Using again \eqref{eq:diese1}, we get that
$$\lim_{k\to\infty}\Ssym{Eu_k}(\R^n) =\lim_{k\to\infty}\sum_{i\in \N} \Ssym{Eu_k}(T_i) =\sum_{i\in \N} \Ssym{Eu}(T_i)=\Ssym{Eu}(\R^n).$$
Finally, since $\Ssym{Eu}(\partial\O)= |Eu|(\partial\O)=0$, we deduce that $\Ssym{Eu_k}(\O) \to \Ssym{Eu}(\O)$.
\end{proof}

\section{Schatten-mass approximation of vector valued functions of bounded variation}

\begin{defn}
A function $u:\R^n \to \R^m$ is {\it piecewise constant} if there exists a countable family $\{P_i\}_{i \in \N}$ of polyhedral sets with pairwise disjoint interiors such that  $\R^n= \bigcup_{i\in \N} P_i$ and vectors $\{c_i\}_{i \in \N}$ in $\R^m$ such that for all $i \in \N$,
$$u(x)=c_i \quad \text{ for all }x \in P_i.$$
We denote by $PC(\R^n;\R^m)$ the space of all such piecewise constant functions.
\end{defn}

\begin{rem}\label{rem:pwcr-jump}
Note that if $u:\R^n \to \R^m$ is piecewise constant, then its gradient, $Du$ is a pure jump measure concentrated on $\bigcup_{i,j \in \N} (\partial P_i \cap \partial P_j)$. Moreover,
$$Du\res (\partial P_i \cap \partial P_j)= \left((u|_{P_j})^+- (u|_{P_i})^-\right) \otimes \nu \HH^{n-1} \res (\partial P_i \cap \partial P_j),$$
where $\nu$ is the unit vector orthogonal to $\partial P_i \cap \partial P_j$ oriented from $P_i$ to $P_j$, and $(u|_{P_j})^+$ (resp. $ (u|_{P_i})^-$) is the trace of $u|_{P_j}$ (resp. $ u|_{P_i}$) on $\partial P_i \cap \partial P_j$. In particular, the variation measure of $Du$ with respect to the Frobenius and Schatten norms coincide, i.e. $|Du|=\Svar{Du}$ as measures in $\R^n$.
\end{rem}

Our second main result is the following approximation of bounded variation vector fields.

\begin{thm}\label{thm:BV}
Let $\O$ be a bounded open set of $\R^n$ with Lipschitz boundary and $u \in BV(\O;\R^m)$. Then, there exists a sequence $\{u_k\}_{k \in \N}$ in $PC(\R^n;\R^m)$ such that
$$
\begin{cases}
u_k \to u \quad \text{strongly in }L^1(\O;\R^m),\\
\Svar{Du_k}(\O) \to \Svar{Du}(\O).
\end{cases}
$$
\end{thm}

The previous approximation result holds for the so-called Schatten-mass $\Svar{Du}$ of $Du$, and not for the usual total variation mass $|Du|$ as the following counterexample shows (see \cite[Remark 22]{AABU} and \cite[Proposition 4]{KristensenRindler}).

\begin{example}\label{ex}{\rm
Let $n=m=2$, $\O=Q=(0,1)^2$ is the unit cube and $u(x)=x$, so that $Du={\rm Id}$. Assume that there exists a sequence of piecewise constant functions $\{u_k\}_{k \in \N}$ such that $u_k \to u$ strongly in $L^1(Q;\R^2)$ and $|Du_k|(Q) \to |Du|(Q)$. Since $\Svar{\cdot}$ is a continuous and positively one-homogeneous function, Reshetnyak continuity Theorem (\cite[Theorem 2.39]{AFP}) ensures that $\Svar{Du_k}(Q) \to \Svar{Du}(Q)$. But since $u_k$ is piecewise constant, Remark \ref{rem:pwcr-jump} shows that
$$\frac{dDu_k}{d|Du_k|}=\alpha_k \otimes \beta_k$$
for some Borel functions $\alpha_k$, $\beta_k:\R^2 \to \R^2$, so that
$$\Svar{Du_k}=\Svar{\left(\frac{dDu_k}{d|Du_k|}\right)}|Du_k|=\Svar{\alpha_k \otimes \beta_k} |Du_k|=|\alpha_k \otimes \beta_k| |Du_k|= |Du_k|.$$
Consequently, $\Svar{Du}(Q)=|Du|(Q)$ which is not possible since $|Du|(Q)=|{\rm Id}|=\sqrt{2}$ while $\Svar{Du}(Q)=\Svar{{\rm Id}}=2$. 
}
\end{example}

Once more, the previous example can be extended to show that it is not possible to approximate $BV$ vector fields by piecewise constant functions in mass with respect to a strictly convex variation. Moreover, our Theorem \ref{thm:BV} is optimal in the sense that the Schatten-1 norm is the only possible norm coinciding with the Frobenius norm on rank-one  matrices for which such a density result holds.

\begin{prop}\label{prop:uniqueBV}
Let $N$ be a norm over $\mathbb M^{m \times n}$ such that $N(a \otimes b)=|a \otimes b|$ for all $a\in \R^m$ and $b \in \R^n$. Assume that for all bounded open set $\O \subset \R^n$ with Lipschitz boundary and all $u \in BV(\O;\R^m)$, there exists a sequence $\{u_k\}_{k \in \N}$ in $PC(\R^n;\R^m)$ such that
$$
\begin{cases}
u_k \to u \quad \text{strongly in }L^1(\O;\R^m),\\
N(Du_k)(\O) \to N(Du)(\O).
\end{cases}
$$
Then $N=\Svar{\cdot}$.
\end{prop}

\begin{proof}
Let $A \in \mathbb M^{m \times n}$ be a fixed matrix and consider the linear function $u(x)=Ax$. By Theorem \ref{thm:BV}, there exists a sequence  $\{u_k\}_{k \in \N}$ in $PC(\R^n;\R^m)$ such that $u_k \to u$ strongly in $L^1(\O;\R^m)$ and $\Svar{Du_k}(\O) \to \Svar{Du}(\O)$. Since $Du_k \wto Du$ weakly* in $\M(\O;\mathbb M^{m \times n})$, Reshetnyak's lower semicontinuity theorem (see \cite[Theorem 2.38]{AFP}) shows that
$$N(Du)(\O) \leq \liminf_{k \to \infty}N(Du_k)(\O).$$
But since, by \eqref{eq:S1=}, $N$, $|\cdot|$ and $\Svar{\cdot}$ coincide on rank-one matrices and, by Remark \ref{rem:pwcr-jump}, $\frac{dDu_k}{d|Du_k|}$ has rank one, we deduce that
$$N(Du) \leq \liminf_{k \to \infty} \Svar{Du_k}(\O)=\Svar{Du}(\O).$$
Recalling that $Du=A$, we get that $N(A) \leq \Svar{A}$. Inverting the role of $\Svar{\cdot}$ and $N$ (for which the approximation property holds by assumption), we infer that $\Svar{A} \leq N(A)$.
\end{proof}

As an immediate consequence of Theorem \ref{thm:BV}, we get the following relaxation result which is the counterpart of Corollary \ref{relBD} in the $BV$ setting.
\begin{cor}
Let $\Omega\subset\R^n$ be open, bounded, with Lipschitz boundary. For all $u \in L^1(\O;\R^m)$, define
$$\overline G(u):=\inf\left\{\liminf_{k \to \infty} |Du_k|(\O) : \ u_k \in PC(\R^n;\R^m), \; u_k \to u \text{ in }L^1(\O;\R^m)\right\}.$$
Then
$$\overline G(u)=
\begin{cases}
\Svar{Du}(\Omega) & \text{ if }u \in BV(\O;\R^m),\\
+\infty & \text{ otherwise}.
\end{cases}$$
\end{cor}

Next we turn to the proof of Theorem \ref{thm:BV}.
\begin{proof}[Proof of Theorem \ref{thm:BV}]
Arguing as in Step 1 of the proof of Theorem \ref{thm:BD}, 
there is no loss of generality to suppose that $u$ is continuous and compactly supported in $\R^n$, and affine on each $n$-simplex $T_i$ of a triangulation of $\R^n$. We now modify $u$ on each of these $n$-simplexes.

\medskip

Let $T$ be a $n$-simplex and $u(x)=Ax+b$ where $A \in \Mmn$ and $b \in \R^m$. We consider the polar decomposition $A=RU$ where $U=\sqrt{A^TA} \in \Ms$ and $R \in O(m,n)$ is an orthogonal matrix. 

Let $\lambda_1,\ldots,\lambda_n \geq 0$ be the eigenvalues of $U$, i.e. the singular values of $A$, and let $e_1,\ldots,e_n$ be the associated eigenvectors which form an orthonormal basis of $\R^n$. We define the piecewise constant function
$$u^T_k(x):=\sum_{i=1}^n \lambda_i s_k(x\cdot e_i) R e_i +b,$$
where $s_k$ has been defined in \eqref{sk},
so that, using the spectral decomposition $U=\sum_{i=1}^n \lambda_i e_i \otimes e_i$,
$$u^T_k(x) \to \sum_{i=1}^n \lambda_i (x\cdot e_i) R e_i +b=Ax+b\quad \text{ uniformly with respect to }x \in \R^n.$$
Next
$$Du^T_k\res T  =  \sum_{\ell \in \mathbb Z}  \sum_{i=1}^n \frac{\lambda_i R(e_i \otimes e_i)}{k} \HH^{n-1}\res \left\{x \in T : \; x\cdot e_i=\frac{\ell}{k}\right\}.$$
Since, for fixed $k \in \N$, the measures 
$$\HH^{n-1}\res \left\{x \in T : \; x\cdot e_i=\frac{\ell}{k}\right\}, \quad i \in \N, \, \ell \in \mathbb Z,$$
are mutually singular (because $\{e_1,\ldots,e_n\}$ forms an orthonormal basis of $\R^n$) and $\Svar{\lambda_i R(e_i\otimes e_i)}=|\lambda_i| \Svar{R(e_i\otimes e_i)}=|\lambda_i| |e_i\otimes e_i|=\lambda_i$ (because $R$ is an orthogonal matrix and the singular values of $A$ are nonnegative), we infer that
\begin{eqnarray*}
\Svar{Du^T_k}(T) & = & \sum_{\ell \in \mathbb Z}  \sum_{i=1}^n \frac{\lambda_i}{k} \HH^{n-1}\left(\left\{x \in T : \; x\cdot e_i=\frac{\ell}{k}\right\}\right)\\
& \to &\left(\sum_{i=1}^n \lambda_i\right) \LL^n(T)=\Svar{A} \LL^n(T)=\Svar{Du}(T).
\end{eqnarray*}

Applying this construction in each $n$-simplex $T_i$, we set
$$u_k:=\sum_{i \in \N} u_k^{T_i} {\bf 1}_{T_i},$$
which defines a sequence $\{u_k\}_{k \in \N}$ of piecewise constant functions in $\R^n$ such that $u_k \to u$ uniformly in $\R^n$ and $\Svar{Du_k}(T_i)\to \Svar{Du}(T_i)$ for all $i \in \N$. Note that since, $u$ has compact support in $\R^n$, we get that 
\begin{equation}\label{eq:diese}
\#\{i \in \N : \; T_i \cap {\rm Supp}(u)\neq \emptyset\}<\infty, \quad \sup_{k \in \N}\#\{i \in \N : \; T_i \cap {\rm Supp}(u_k)\neq \emptyset\}<\infty.
\end{equation}

It remains to estimate the measure $\Svar{Du_k}$ on the common interface $S_{ij}=\partial T_i \cap \partial T_j$ of two adjacent $n$-simplexes $T_i$ and $T_j$. To this aim, we observe that, by construction, for all $i \in \N$,
$$\|u_k-u\|_{L^\infty(T_i;\R^m)} \leq \frac{|A_i|}{k}$$
and thus, owing to \eqref{eq:diese},
$$\|u_k-u\|_{L^\infty(T_i;\R^m)} \leq \frac{C}{k}$$
for some constant $C>0$ independent of $k$. As a consequence of the continuity of $u$, we deduce that for all $x \in S_{ij}$
$$|u_k^+(x)- u_k^-(x)| \leq \frac{2C}{k},$$
where $u_k^-=u^{T_i}_k|_{S_{ij}}$ and $u_k^+=u^{T_j}_k|_{S_{ij}}$ denote the one-sided traces of $u_k$ on both sides of $S_{ij}$. On the other hand, the jump formula yields
$$Du_k \res S_{ij}= (u_k^+-u_k^-)\otimes \nu \HH^{n-1}\res S_{ij},$$
where $\nu$ is the normal vector to $S_{ij}$ oriented from $T_i$ to $T_j$. As a consequence,
$$\Svar{Du_k}(S_{ij})= \int_{S_{ij}} \Svar{(u_k^+-u_k^-)\otimes \nu}\, d \HH^{n-1} = \int_{S_{ij}} |(u_k^+-u_k^-)\otimes \nu| \, d\HH^{n-1} \leq  \frac{2C}{k}\HH^{n-1}(S_{ij}) \to 0$$
as $k \to \infty$. Using again \eqref{eq:diese}, we get that
$$\lim_{k\to\infty}\Svar{Du_k}(\R^n) =\lim_{k\to\infty}\sum_{i\in \N} \Svar{Du_k}(T_i) = \sum_{i\in \N} \Svar{Du}(T_i)=\Svar{Du}(\R^n).$$
Finally, since $\Svar{Du}(\partial\O)= |Du|(\partial\O) =0$, we deduce that $\Svar{Du_k}(\O) \to \Svar{Du}(\O)$.
\end{proof}

\section{Concluding remarks and open problems}

It is to be expected that some density result with singular objects holds for measures satisfying a general linear PDE constraint as in \cite{DPR} (the so-called $\mathcal A$-free measures). The results presented in this work correspond to the particular cases $\mathcal A={\rm curl}$ (in the vectorial $BV$ case) and $\mathcal A={\rm curl\, curl}$ (in the $BD$ case). 

\medskip

Another relevant example, e.g. in materials science, is the divergence constraint (see \cite{BIR2}). Given a bounded open set, we define the space
$$DM(\O)=\{ \sigma \in \M(\O;\Ms) : \ {\rm div}\sigma=0 \text{ in }\mathcal D'(\O;\R^n)\}.$$
The wave cone associated to this differential constraint (see e.g. \cite[Section 2.2]{BIR2}) is given by
$$\Lambda_{\rm div}=\{ A \in \Ms : \; {\rm det}(A)=0\}.$$
In dimension $n=2$, it is known that (provided $\O$ is smooth enough and simply connected), for all $\sigma \in DM(\O)$, there exists a function $u \in BH(\O)$ (which means that $u \in W^{1,1}(\O)$ and $D^2 u \in \M(\O; \Ms)$) such that $\sigma={\rm cof}(D^2 u)$. In mechanical language, the function $u$ is sometimes referred to as the Airy function. Applying \cite[Theorem 21]{AABU} (see also \cite[Theorem 2.2]{ABC}), there exists a sequence $\{u_k\}_{k \in \N}$ of continuous and piecewise affine functions such that $u_k \to u$ in $L^\infty(\O)$ and $\Svar{D^2 u_k}(\O) \to \Svar{D^2 u}(\O)$. Since $u_k$ is continuous and piecewise affine, it follows that $\nabla u_k \in SBV(\O;\R^2)$ and
$$D^2 u_k= a_k \nu_{J_{\nabla u_k}}\otimes\nu_{J_{\nabla u_k}} \HH^1 \res J_{\nabla u_k},$$
where $a_k: J_{\nabla u_k} \to \R$ is a Borel function and $\nu_{J_{\nabla u_k}}$ is the approximated normal to the jump $J_{\nabla u_k}$ of $\nabla u_k$. Defining $\sigma_k={\rm cof}(D^2 u_k)$, we get that $\sigma_k \in DM(\O)$,
$$\sigma_k = a_k  \tau_{J_{\nabla u_k}}\otimes\tau_{J_{\nabla u_k}} \HH^1 \res J_{\nabla u_k},$$
where $\tau_{J_{\nabla u_k}}=R \nu_{J_{\nabla u_k}}$ is an approximate tangent vector to $J_{\nabla u_k}$, and $R$ is the rotation matrix
$$R=
\begin{pmatrix}
0 & -1\\
1 & 0
\end{pmatrix}.
$$
Since a matrix $A \in \mathbb M^{2 \times 2}_{\rm sym}$ share the same eigenvalues with ${\rm cof}(A)$, we get that $\Svar{\sigma_k}(\O) \to \Svar{\sigma}(\O)$ and $\frac{d\sigma_k}{d|\sigma_k|} =\pm \tau_{J_{\nabla u_k}}\otimes\tau_{J_{\nabla u_k}} \in \Lambda_{\rm div}$ $|\sigma_k|$-a.e. in $\O$. 

Let us denote by $|\cdot|_{\dive}:=(|\cdot|+I_{\Lambda_{\dive}})^{**}$. As already evidenced in \cite{Bouchitte} and \cite[Formula (1.10)]{BIR2}, one has
$$|\cdot|_{\dive}=|\cdot|_S \quad \text{ in }\mathbb{M}^{2\times 2}_{\rm sym}.$$
Following again \cite{BIR2,Bouchitte}, we expect a similar density result to hold in dimension $n=3$ with for all $A \in \mathbb{M}^{3\times 3}_{\rm sym}$,
$$|A|_{\dive}=
\begin{cases}
\sqrt{(|\lambda_1|+|\lambda_2|)^2+\lambda_3^2} & \text{ if }|\lambda_1|+|\lambda_2|\leq |\lambda_3|,\\
\frac{1}{\sqrt 2}(|\lambda_1|+|\lambda_2|+|\lambda_3|) & \text{ if }|\lambda_1|+|\lambda_2|> |\lambda_3|,
\end{cases}$$
where $\lambda_1$, $\lambda_2$ and $\lambda_3$ are the eigenvalues of $A$ ordered as singular values $|\lambda_1|\leq |\lambda_2|\leq |\lambda_3|$.

\bigskip

In view of the structure of the singular part of $\mathcal{A}$-free Radon measures \cite{Alberti,DPR} and of our results Theorems \ref{thm:BD} and \ref{thm:BV}, one may expect the following general statement to be true:

\medskip

{\it Let $\mathcal A: \mathcal D'(\R^n;\R^m) \to \mathcal D'(\R^n;\R^d)$ be a linear differential operator and $\Lambda_{\mathcal A}$ be its associated wave cone. For every $z \in \R^m$, we define
$$|z|_{\mathcal A}:=(| \cdot | + I_{\Lambda_{\mathcal A}})^{**}(z),$$
where $|\cdot |$ is the Euclidean norm over $\R^m$ and $ I_{\Lambda_{\mathcal A}}$ is the indicator function of the set $\Lambda_{\mathcal A}$. Then, $|\cdot|_{\mathcal A}$ is a norm over $\R^m$. Moreover if $\O \subset \R^n$ is a bounded open set with Lipschitz boundary, for every $\mu \in \M(\O;\R^m)$ satisfying $\mathcal A \mu=0$ in $\mathcal D'(\O;\R^d)$, there exists a sequence $\{\mu_k\}_{k \in \N}$ in $\M(\O;\R^m)$ such that 
$$\begin{cases}
\mathcal A \mu_k=0& \text{in }\mathcal D'(\O;\R^d),\\
|\mu_k|_{\mathcal A}(\O) \to |\mu|_{\mathcal A}(\O),\\
\mu_k \perp \LL^n,&\\
\frac{d\mu_k}{d|\mu_k|} \in \Lambda_{\mathcal A} & |\mu_k|\text{-a.e. in }\O.
\end{cases}$$
The construction in this case would be of course more delicate since, depending on the order and on the form of the differential operator, compatibility constraints on the support of the $\mu_k$'s may be required. This goes beyond the scopes of the present paper and will be left to future investigations.}

\section*{Acknowledgements}

This work was supported by a public grant from the Fondation Math\'ematique Jacques Hadamard. FI acknowledges support of a International Emerging Actions project of the CNRS.

\bibliographystyle{siam}
\bibliography{biblio}

\end{document}